\documentclass{amsart}

\usepackage{amsthm,amsmath,amsfonts}
\usepackage[all]{xy}
\usepackage{color}
\usepackage[colorlinks=true, linkcolor=blue, citecolor=blue]{hyperref}
\usepackage{tikz-cd}

\numberwithin{equation}{section} 

\newtheorem{theorem}{Theorem}[section]
\newtheorem{lemma}[theorem]{Lemma}

\newtheorem{corollary}[theorem]{Corollary}

\theoremstyle{definition}
\newtheorem{definition}[theorem]{Definition}
\newtheorem{example}[theorem]{Example}

\newtheorem*{ackno}{Acknowledgements}

\theoremstyle{remark}
\newtheorem{remark}[theorem]{Remark}

\newcommand{\sA}{\mathcal{A}}
\newcommand{\sE}{\mathcal{E}}
\newcommand{\sJ}{\mathcal{J}}
\newcommand{\sM}{\mathcal{M}}
\newcommand{\Oh}{\mathcal{O}}
\newcommand{\sF}{\mathcal{F}}
\newcommand{\sU}{\mathcal{U}}

\newcommand{\K}{\mathbb{K}}
\newcommand{\Z}{\mathbb{Z}}
\newcommand{\C}{\mathbb{C}}

\newcommand{\na}{\nabla}

\renewcommand{\bar}{\overline}

\newcommand{\de}{\partial}
\newcommand{\debar}{\overline{\partial}}

\newcommand{\Def}{\operatorname{Def}}

\newcommand{\Tot}{\operatorname{Tot}}

\newcommand{\Hom}{\operatorname{Hom}}

\newcommand{\Tr}{\operatorname{Tr}}
\newcommand{\Ext}{\operatorname{Ext}}
\newcommand{\At}{\operatorname{At}}

\newcommand{\HOM}{{{\mathcal H}om}}
\newcommand{\DER}{{{\mathcal D}er}}

\begin{document}

\title{Connections and $L_{\infty}$ liftings of semiregularity maps}

\author{Emma Lepri}
\address{\newline
Universit\`a degli studi di Roma La Sapienza,\hfill\newline
Dipartimento di Matematica  Guido
Castelnuovo,\hfill\newline
P.le Aldo Moro 5,
I-00185 Roma, Italy.\medskip}
\email{emma.lepri@uniroma1.it}

\author{Marco Manetti}
\email{manetti@mat.uniroma1.it}
\urladdr{www.mat.uniroma1.it/people/manetti/}
\date{May 22, 2021}


\begin{abstract} Let $\sE^*$ be a finite complex of locally free sheaves on a complex manifold $X$. We prove that to every connection of type $(1,0)$ on $\sE^*$ it is canonically associated an $L_{\infty}$ morphism 
\[ g\colon A^{0, *}_X(\HOM^*_{\Oh_X}(\sE^*,\sE^*))\rightsquigarrow \dfrac{A^{*,*}_X}{A^{\ge 2,*}_X}[2]\] 
that lifts the 1-component of the Buchweitz-Flenner semiregularity map.	An application to deformations of coherent sheaves on projective manifolds is given.
\end{abstract}

\maketitle

\section{Introduction}

In the remarkable paper \cite{BF}, Buchweitz and Flenner generalise the classical Severi-Kodaira-Bloch's  semiregularity map to any coherent sheaf $\sF$ of locally finite projective dimension on an arbitrary complex space $X$. Their description involves exterior powers of the cotangent complex and when $X$ is a smooth complex manifold it
reduces to a map
\[ \sigma\colon  \Ext^2_X(\sF,\sF)\to \prod_{q\ge 0}H^{q+2}(X,\Omega_X^q)\,\]
defined in terms of the  Atiyah class $\At(\sF)\in \Ext^1_X(\sF,\sF\otimes \Omega^1_X)$  of $\sF$, see also \cite{KM}.
More precisely,  the Yoneda pairing 
\[ \Ext^i_X(\sF,\sF\otimes \Omega^i_X)\times \Ext^j_X(\sF,\sF\otimes \Omega^j_X)\to 
\Ext^{i+j}_X(\sF,\sF\otimes \Omega^{i+j}_X),\qquad (a,b)\mapsto a\circ b,\]
allows to define the exponential of the opposite of the Atiyah class
\[ \exp(-\At(\sF))\in \prod_{q\ge 0}\Ext^q_X(\sF,\sF\otimes \Omega^q_X)\,.\]
Since $X$ is assumed smooth, every coherent sheaf $\sF$ has locally finite projective dimension 
and then there are defined the trace maps 
\[ \Tr\colon \Ext_X^i(\sF,\sF\otimes \Omega^j_X)\to H^i(X,\Omega^j_X),\qquad i,j\ge 0\,.\]

It is worth  recalling that, as proved by Atiyah for vector bundles and by Illusie in the general case \cite{At,Ill}, when $X$ is a projective manifold, then with respect to the Hodge decomposition in cohomology, 
the trace of the above exponential gives the Chern character of $\sF$, cf. \cite[p.137]{BF}:
\[ \operatorname{ch}(\sF)=\Tr( \exp(-\At(\sF)))\,.\]
 
The Buchweitz-Flenner semiregularity map is defined by the formula:
\[ \sigma=\sum_{q\ge 0}\sigma_q\colon \Ext^2_X(\sF,\sF)\to \prod_{q\ge 0}H^{q+2}(X,\Omega_X^q),\qquad 
\sigma(c)=\Tr(\exp(-\At(\sF))\circ c)\,.\]
The name semiregularity is motivated by the fact that when $\sF=\Oh_Z$ for a locally complete intersection $Z$ of codimension $q$, then the map $\sigma_{q-1}$ is the same as the classical semiregularity map defined by Bloch \cite{bloch}.

The semiregularity map is important both in the variational Hodge conjecture and in deformation theory: we refer to \cite{bloch} for a discussion about the first subject, 
while for the second we recall that  $\Ext^2_X(\sF,\sF)$ 
is the obstruction space for the functor of deformations of $\sF$. Moreover,  a classical result by Mukai and Artamkin \cite{Arta,FMM,DMcoppie}, asserts that the $0$th component of the semiregularity map
\[ \sigma_0\colon \Ext^2_X(\sF,\sF)\to H^{2}(X,\Oh_X),\qquad \sigma_0(c)=\Tr(c),\]
annihilates obstructions; in other words the kernel of the trace map $\Tr\colon \Ext^2_X(\sF,\sF)\to H^{2}(X,\Oh_X)$ is an obstruction space for $\sF$.

One of the main results of \cite{BF} is that if the Hodge to de Rham spectral sequence of $X$ degenerates at $E_1$, then every curvilinear (and also semitrivial, see \cite{Man})  obstruction is annihilated 
by the semiregularity map. However Buchweitz and Flenner left unanswered the question of whether the semiregularity map annihilates every obstruction.   

As suggested in \cite{BF}, and then clarified in \cite{FioMan,semireg2011,Pri}, the correct way to interpret the semiregularity map is as the obstruction map of a morphism of deformation theories, with target in a product of  (formal) intermediate Jacobians: this implies in particular that from the point of view of deformation theory it is more appropriate, for every $q$,  to consider the composition 
\[  \tau_q\colon \Ext_X^2(\sF,\sF)\xrightarrow{\sigma_q}H^{q+2}(X,\Omega_X^q)=H^{2}(X,\Omega_X^q[q])\xrightarrow{i_q} \mathbb{H}^{2}(X,\Omega_X^{\le q}[2q]),\]
where $\Omega^{\le q}_X=(\oplus_{i=0}^q\Omega_X^i[-i],\de)$ is the truncated holomorphic de Rham complex and $i_q$ is induced by the inclusion of complexes $\Omega_X^{q}[q]\subset \Omega_X^{\le q}[2q]$.
Clearly, when the Hodge to de Rham spectral sequence of $X$ degenerates at $E_1$ the second map is injective and therefore $\tau_q$ and $\sigma_q$ have the same kernel.

Finally, the paper \cite{Pri} by Pridham contains a proof that the maps $\tau_q$ annihilate every obstruction, for every coherent sheaf on a smooth manifold, while the analogous result for the classical Bloch semiregularity map was previously proved, under some mild additional assumption, in 
\cite{semireg2011}.
Pridham  works in the framework of homotopy homogeneous functors in the category of simplicial commutative algebras in order to construct a morphism of deformation theories inducing the semiregularity map. 

The aim of this paper is to construct, by elementary methods, an explicit morphism of deformation theories 
lifting $\tau_1$ in the framework of differential graded Lie algebras. 
We expect that this explicit approach should also work  for the lifting of $\tau_q$, $q>1$, although this  appears, at this moment, extremely complicated from the computational point of view.  

Roughly speaking, by a nowadays well established theory, the category of deformation problems over a field of characteristic $0$ is equivalent to the homotopy category of DG-Lie algebras over the same field: two DG-Lie algebras are homotopy equivalent if they are quasi-isomorphic or, equivalently, if they have isomorphic $L_{\infty}$ minimal models, and 
every morphism in the homotopy category 
can be represented by  an $L_{\infty}$ morphism. The passage from homotopy classes of DG-Lie algebras to the standard, and more geometric, description of deformation problems is given by taking solutions of 
the Maurer-Cartan equation modulus gauge action, see e.g. \cite{Man,LMDT} and references therein.

It is well known that the deformation theory of a coherent sheaf $\sF$ is controlled by $R\Hom(\sF,\sF)$ (see e.g. \cite{FIM,HL,Meazz}), considered as an element  
in the homotopy category of differential graded Lie algebras: if $\sF$ admits a finite locally free resolution (e.g. if $X$ is a smooth projective manifold) 
\begin{equation} \label{equ.resolutionintro}
0\to \sE^{-n}\to \cdots\to \sE^0\to \sF\to 0
\end{equation}
then a representative of  $R\Hom(\sF,\sF)$ is given by the Dolbeault complex 
$A_X^{0,*}(\HOM^*_{\Oh_X}(\sE^*,\sE^*))$.

Recall that the quasi-isomorphic complexes $\sE^*$ and $\sF$ have the same Atiyah class and that 
the hypercohomology of $\Omega_X^{\le 1}[2]$ is computed by the truncated de Rham complex $A_X^{*,*}/A_X^{\ge 2,*}[2]$. As in the case of locally free sheaves (see e.g. \cite{ABTT,At}), the Atiyah class of $\sE^*$ can be computed by using connections of type $(1,0)$: we describe this construction in Section~\ref{sec.connection}.

Then the main results of this paper are:

\begin{theorem}[=Corollary~\ref{cor.thm.main2}] Let $\sF$ be a coherent sheaf on a complex manifold $X$ admitting a finite locally free resolution $\sE^*$ as in
\eqref{equ.resolutionintro}. Then the choice of a connection of type $(1,0)$ on $\sE^i$ for every $i$  gives an explicit lifting of 
\[  \tau_1\colon \Ext_X^*(\sF,\sF)\to \mathbb{H}^{*}(X,\Omega_X^{\le 1}[2])\]
to an $L_{\infty}$ morphism 
\[ g\colon A_X^{0,*}(\HOM^*_{\Oh_X}(\sE^*,\sE^*))\rightsquigarrow \frac{A_X^{*,*}}{A_X^{\ge 2,*}}[2]\,.\]
\end{theorem}

In the above theorem, by the  term lifting we mean that 
the linear component 
\[ g_1\colon A_X^{0,*}(\HOM^*_{\Oh_X}(\sE^*,\sE^*))\to \frac{A_X^{*,*}}{A_X^{\ge 2,*}}[2]\,\]
is a morphism of complexes inducing $\tau_1$  in cohomology. A brief review of $L_{\infty}$ morphisms between differential graded Lie algebras will be given in Section~\ref{sec.Linfinito}.

Since $X$ is  smooth by assumption, according to Hilbert's syzygy theorem (see e.g. \cite[V.3.11]{Ko}), if a coherent sheaf $\sF$ on $X$ admits a locally free resolution, then it  also admits a finite locally free resolution. The following result is an almost immediate consequence of the above theorem.

\begin{theorem}[=Corollary~\ref{cor.main2}] Let $\sF$ be a coherent sheaf on a complex manifold $X$ admitting a locally free resolution. Then every obstruction to the deformations of $\sF$ belongs to the kernel of the map
\[  \tau_1\colon \Ext_X^2(\sF,\sF)\to \mathbb{H}^{2}(X,\Omega_X^{\le 1}[2]).\]
If the Hodge to de Rham spectral sequence of $X$ degenerates at $E_1$, then  
every obstruction to the deformations of $\sF$ belongs to the kernel of the map
\[  \sigma_1\colon \Ext_X^2(\sF,\sF)\to {H}^{3}(X,\Omega_X^{1}),\qquad \sigma_1(a)=-\Tr(\At(\sF)\circ a).\]
\end{theorem}

It is worth  pointing out that our approach is almost entirely algebraic and that 
the above corollary holds also over every smooth separated scheme of finite type over a field of characteristic 0 \cite{lepri}: a proof of this purely algebraic analogous result is outlined in Section~\ref{sec.outline}.

\medskip

\subsection{Notation}  If $V=\oplus V^i$ is either a graded vector space or a graded sheaf, we denote by $\bar{v}$ the degree of a homogeneous element $v\in V$. For every integer $p$ the symbol $[p]$ denotes the shift functor, defined by $V[p]^i=V^{p+i}$.

For a  complex manifold $X$, its de Rham complex  is denoted by $(A^{*,*}_X,d=\de+\debar)$ and  the subcomplexes of the Hodge filtrations by $A^{\ge p,*}_X$. 
The holomorphic de Rham complex of $X$ is denoted by $\Omega_X^*=(\oplus_{i\ge 0}\Omega_X^i[-i],\de)$.

\bigskip
\section{Connections and Atiyah classes}
\label{sec.connection}

The theory of connections of type $(1,0)$ on holomorphic vector bundles \cite{ABTT,At,Ko} extends without difficulty to every complex of locally free sheaves. For simplicity of exposition we consider here only the case of finite complexes, which is completely sufficient for our applications. 	
	
Let $X$ be a complex manifold and let  
\[ \sE^*\colon\qquad 0\to \sE^p\xrightarrow{\,\delta\,}\sE^{p+1}\xrightarrow{\,\delta\,}\cdots\xrightarrow{\,\delta\,}\sE^{q}\to 0,\qquad\quad p,q\in \Z,\quad \delta^2=0,\]
be a fixed finite complex of locally free sheaves of $\Oh_X$-modules. We denote by 
$\HOM^*_{\Oh_X}(\sE^*,\sE^*)$ the graded sheaf of $\Oh_X$-linear endomorphisms of $\sE^*$: 
\[ \HOM^*_{\Oh_X}(\sE^*,\sE^*)=\bigoplus_i \HOM^i_{\Oh_X}(\sE^*,\sE^*),\qquad 
\HOM^i_{\Oh_X}(\sE^*,\sE^*)=\prod_j\HOM_{\Oh_X}(\sE^j,\sE^{i+j})\,.\]
Then $\HOM^*_{\Oh_X}(\sE^*,\sE^*)$ is a sheaf of locally free DG-Lie algebras over $\Oh_X$, 
with the bracket equal to the graded commutator 
\[ [f,g]=fg-(-1)^{\bar{f}\,\bar{g}}gf\]
and the differential given by 
\[f\mapsto [\delta,f]=\delta f-(-1)^{\bar{f}}f\delta\,.\]

For every $a,b,r$ 
denote by $\sA^{a,b}_X(\sE^r)\simeq \sA^{a,b}_X\otimes_{\Oh_X}\sE^r$ the sheaf of differential forms of type $(a,b)$ with coefficients in $\sE^r$, and by \[ \debar\colon \sA^{a,b}_X(\sE^r)\to \sA^{a,b+1}_X(\sE^r),\qquad \debar(\phi\cdot e)=\debar(\phi)\cdot e\] the Dolbeault differential.

	We consider  
	\[ \sA^{*,*}_X(\sE^*)=\bigoplus_{a,b,r}\sA^{a,b}_X(\sE^r)\]
	as a graded sheaf on $X$, where the elements of $\sA^{a,b}_X(\sE^r)$ have degree $a+b+r$.
	Unless otherwise specified, when we write  $\phi\cdot e\in \sA^{*,*}_X(\sE^*)$ we  mean that $\phi$ is a differential form and  $e$ is a holomorphic section of $\sE^*$, or a germ of one.  
Sometimes, for notational simplicity, whenever $\psi\in \sA^{*,*}_X$ and 
$h\in \sA^{*,*}_X(\sE^*)$ we denote by $\psi\cdot h\in \sA^{*,*}_X(\sE^*)$ their product, namely the bilinear extension of $\psi\cdot (\phi\cdot e)=(\psi\wedge \phi)\cdot e$, $\phi\in \sA^{*,*}_X$, $e\in\sE^*$.

	In accordance with the Koszul sign rule, the differential $\delta$ can be extended to a differential 
	\[ \delta\colon \sA^{a,b}_X(\sE^r)\to \sA^{a,b}_X(\sE^{r+1}),\qquad \delta(\phi\cdot e)=
	(-1)^{\bar{\phi}}\,\phi\cdot \delta(e)\,.\]
	We have that $\debar^2=\delta^2=0$ and $[\debar,\delta]=\debar\delta+\delta\debar=0$, so that
	$\debar+\delta$ is a differential in $\sA^{*,*}_X(\sE^*)$.
	
	Therefore the space of $\C$-linear morphisms of sheaves 
	\[ \Hom_{\C}^*(\sA^{*,*}_X(\sE^*),\sA^{*,*}_X(\sE^*))\]
 carries a natural structure of  differential graded associative algebra:  the product is given by composition and the differential is the graded commutator with $\debar+\delta$.
 	
	Denoting by $A^{a,b}_X(-)$ the global sections of $\sA^{a,b}_X(-)$ we have two differential graded subalgebras
	\[ A^{0,*}_X(\HOM^*_{\Oh_X}(\sE^*,\sE^*))\subset A^{*,*}_X(\HOM^*_{\Oh_X}(\sE^*,\sE^*))
	\subset \Hom_{\C}^*(\sA^{*,*}_X(\sE^*),\sA^{*,*}_X(\sE^*)),\]
where for  
	\[ \omega, \eta \in \sA^{*,*}_X,\quad f\in \HOM^*_{\Oh_X}(\sE^*,\sE^*),\quad e\in \sE^*\]
	one has that 
	\[ (\omega\cdot f)(\eta\cdot e)=(-1)^{\bar{f}\bar{\eta}}(\omega\wedge\eta)\cdot f(e)\,,\]
	so that the elements of $A^{a,b}_X(\HOM^n_{\Oh_X}(\sE^*,\sE^*))$ have degree $a+b+n$.
For every $a \in A^{*,*}_X(\HOM^*_{\Oh_X}(\sE^*, \sE^*))$ we have 
		\begin{equation}\label{rem.debar} 
		\debar a=[\debar, a ]
		\end{equation}
where the bracket on the right is intended in the DG-Lie algebra $\Hom_{\C}^*(\sA^{*,*}_X(\sE^*),\sA^{*,*}_X(\sE^*))$. In fact, for $\omega, \eta\in \sA^{*,*}_X$, $f \in \HOM^*_{\Oh_X}(\sE^*, \sE^*)$ and $e\in \sE^*$ we have:
\begin{align*}
		[\debar, \omega \cdot f] (\eta \cdot e) &= \debar ((-1)^{\overline{f} \overline{\eta}} \omega \wedge \eta \cdot f(e)) - (-1)^{\overline{\omega}+ \overline{f}} (\omega \cdot f) (\debar (\eta) \cdot e ) \\
		& = (-1)^{\overline{f} \overline{\eta}} \debar (\omega) \wedge \eta \cdot f(e) + (-1)^{\overline{f} \overline{\eta}+ \overline{\omega}} \omega \wedge \debar ( \eta)  \cdot f(e) - (-1)^{\overline{\omega}+ \overline{f} \overline{\eta}} \omega \wedge \debar (\eta) \cdot f (e) \\
		& = (\debar (\omega)\cdot f)(\eta \cdot e). 
		\end{align*}

The composition product in $A^{*,*}_X(\HOM^*_{\Oh_X}(\sE^*,\sE^*))$ works in the following way:
\[ \omega, \eta\in \sA^{*,*}_X,\quad f,g\in \HOM^*_{\Oh_X}(\sE^*,\sE^*)\]
\[ (\omega\cdot f)(\eta\cdot g)=(-1)^{\bar{f}\bar{\eta}}(\omega\wedge\eta)\cdot fg\,,\]
and the commutator is
\[ [\omega\cdot f,\eta\cdot g]=(\omega\cdot f)(\eta\cdot g)-(-1)^{(\bar{\omega}+\bar{f})
		(\bar{\eta}+\bar{g})}(\eta\cdot g)(\omega\cdot f)=(-1)^{\bar{f}\bar{\eta}}
	(\omega\wedge\eta)\cdot [f,g]\,.\]
The above commutator and the differential $[\delta +\debar, -]=[\delta, -] + \debar$ give 
$A^{*,*}_X (\HOM^*_{\Oh_X}(\sE^*, \sE^*))$ a structure of DG-Lie algebra.
	
\begin{definition} Let $\sE^*$ be a finite complex of locally free sheaves on a complex manifold $X$.
			A connection of type $(1,0)$ on $\sE^*$ is an element $D$ of
		$\Hom_{\C}^1(\sA^{*,*}_X(\sE^*),\sA^{*,*}_X(\sE^*))$ such that:
		\begin{enumerate}
			
	\item $D(\phi\cdot s)=d\phi\cdot s+(-1)^{\bar{\phi}}\phi\cdot D(s)$, with $s\in \sE^r$ and $\phi\in \sA^{*,*}_X$;\smallskip
	
	\item if $s\in \sE^r$ is a holomorphic section then $D(s)\in \sA^{1,0}_X(\sE^r)$.
			
	\end{enumerate}
\end{definition}
The second condition is equivalent to $D = D^{1,0}+ \debar$, with 
$D^{1,0} \colon \sA^{a,b}_X(\sE^r) \to \sA^{a+1,b}_X(\sE^r)$ for every $a,b,r$. 

Basically it is the same as giving a connection which is compatible with the holomorphic structure on every vector bundle of the complex. In particular, connections of type $(1,0)$ always exist.

From now on, let $D=D^{1,0}+ \debar$ be a fixed connection of type $(1,0)$ on a fixed finite complex $\sE^*$ 
of locally free sheaves. Denote by 
\[\nabla  =[D - \debar,-]=[D^{1,0},-]\colon \Hom_{\C}^*(\sA^{*,*}_X(\sE^*),\sA^{*,*}_X(\sE^*))\to \Hom_{\C}^*(\sA^{*,*}_X(\sE^*),\sA^{*,*}_X(\sE^*))\]
the adjoint operator of the $(1,0)$-component of the connection $D$.

As in the classical case, the adjoint operator $[D,-]$ is a connection of type $(1,0)$ in the complex of locally free sheaves $\HOM^*_{\Oh_X}(\sE^*,\sE^*)$. We need this fact only in the weaker form given by the following lemma.

\begin{lemma}\label{lem.operatoreaggiunto} 
In the above situation, 
for every $h\in A^{*,*}_X(\HOM^*_{\Oh_X}(\sE^*,\sE^*))$ we have 
		\[  [D, h],\, \nabla(h) \in A^{*,*}_X(\HOM^*_{\Oh_X}(\sE^*,\sE^*))\,.\]
More precisely, if $\eta \in \sA^{a,b}_X$ and $g \in \HOM^n_{\Oh_X}(\sE^*, \sE^*))$, then
	$ \nabla (\eta \cdot g) \in A^{a+1, b}_X (\HOM^n_{\Oh_X}(\sE^*,\sE^*))$.
\end{lemma}
	
\begin{proof} Since every $\sE^i$ is locally free, we can describe 	$A^{*,*}_X(\HOM^*_{\Oh_X}(\sE^*,\sE^*))$ as the set of morphisms of graded sheaves $h\colon \sA^{*,*}_X(\sE^*)\to \sA^{*,*}_X(\sE^*)$ that are 
$\sA^{*,*}_X$ linear, i.e., $h(f\cdot s)=(-1)^{\bar{h}\,\bar{f}}f\cdot h(s)$, $f\in \sA^{*,*}_X$, $s\in \sA^{*,*}_X(\sE^*)$.

Thus, for every $f\in \sA^{*,*}_X$ and $s\in \sA^{*,*}_X(\sE^*)$ we have
\[\begin{split} 
[D,h](fs)&=(-1)^{\bar{h}\,\bar{f}}D(fh(s))-(-1)^{\bar{h}}h(df\cdot s+(-1)^{\bar{f}}f\cdot D(s))\\
&=(-1)^{\bar{h}\,\bar{f}}(df\cdot h(s)+(-1)^{\bar{f}}fD(h(s)))-(-1)^{\bar{h}+\bar{h}(\bar{f}+1)}df\cdot h(s)-
(-1)^{\bar{h}+\bar{f}+\bar{h}\bar{f}} f\cdot h(D(s))\\
&=(-1)^{(\bar{h}+1)\bar{f}}f\cdot [D,h](s).\end{split}\]
This proves that $[D,h]\in  A^{*,*}_X(\HOM^*_{\Oh_X}(\sE^*,\sE^*))$. According to \eqref{rem.debar} we also have 
\[\nabla(h)=[D,h]-\debar(a)\in A^{*,*}_X(\HOM^*_{\Oh_X}(\sE^*,\sE^*))\,.\]
\end{proof}

	\begin{lemma} In the above notation, define 
	\[u=[D,\debar+\delta]=[\debar+\delta,D] = \nabla (\debar + \delta)\,,\] 
	where the last equality follows by $[\debar,\debar+\delta]=0$. Then 
	$u\in A^{*,*}_X(\HOM^*_{\Oh_X}(\sE^*,\sE^*))$, and more precisely
		\[ u\in A^{1,1}_{X}(\HOM^0_{\Oh_X}(\sE^*,\sE^*))\oplus A^{1,0}_X(\HOM^1_{\Oh_X}(\sE^*,\sE^*))\,.\]
	\end{lemma}

	\begin{proof} 
	
	The element $u\colon \sA^{*,*}_X(\sE^*)\to \sA^{*,*}_X(\sE^*)$ is a morphism of graded sheaves of even degree and we need to show that it is $\sA^{*,*}_X$-linear. By Lemma~\ref{lem.operatoreaggiunto} 
we have $[D, \delta]\in A^{1,0}_X(\HOM^1_{\Oh_X}(\sE^*, \sE^*))$ and we only need to prove that  	
$[D, \debar]$ is  $\sA^{*,*}_X$-linear. 
	
For $f \in \sA^{*,*}_X, s\in \sA^{*,*}_X(\sE^*)$ we have:	
	
\[\begin{split}
[D, \debar](fs)&= D (\debar f \cdot s +(-1)^{\bar{f}}f\cdot \debar s ) 
+ \debar  (df \cdot s +(-1)^{\bar{f}} f D (s) ) \\[3pt]
&= d (\debar f) \cdot s +(-1)^{\bar{f}+1}\debar f\cdot D(s)+(-1)^{\bar{f}}df\cdot \debar s+fD(\debar(s))\\
&\quad+\debar(df)\cdot s+(-1)^{\bar{f}+1}df\cdot \debar s  +(-1)^{\bar f}\debar f\cdot D(s)+f\debar(D(s))  \\
& = f [D, \debar ](s).\end{split}\]
Therefore $[D, \debar]\in A^{1,1}_X(\HOM^0_{\Oh_X}(\sE^*, \sE^*))$.
	\end{proof}

\begin{definition}[Atiyah class]\label{def.atiyahclass}
The Atiyah cocycle $u\in A^{1,*}_X(\HOM^*_{\Oh_X}(\sE^*,\sE^*))$ of the connection $D$ of type $(1,0)$ as above is defined by the formula
\[u=[D,\debar+\delta]=[\debar+\delta,D] = \nabla (\debar + \delta)\,.\]
 
The fact that $[\debar+\delta,u]=\debar(u)+[\delta,u]=0$, i.e., that $u$ is a cocycle 
follows immediately by the Jacobi identity:
\[  [\debar+\delta,u]=[\debar+\delta,[\debar+\delta,D]]=\frac{1}{2}[[\debar+\delta,\debar+\delta],D]=0\,.\]
	
The Atiyah class of $\sE^*$ is the cohomology class of the Atiyah cocycle 
\[ \At(\sE^*)=[u]\in \mathbb{H}^2(A^{1,*}_X(\HOM^*(\sE^*,\sE^*)))=\Ext_X^1(\sE^*,\Omega^1\otimes \sE^*)\,.\]  
	\end{definition}
	~\\

The Atiyah class does not depend on the choice of the connection of type $(1,0)$: two such connections $D, D'$ differ by $a \in A^{1,0}_X(\HOM^0_{\Oh_X}(\sE^*,\sE^*))$, so that 
	\begin{align*}
	& u'= [D', \debar + \delta] = [D + a, \debar + \delta ] = [D, \debar + \delta] + [a, \debar + \delta] = u + [\delta, a] + \debar a,
	\end{align*}
	and $u$ and $u'$ represent the same cohomology class.

It is straightforward to verify that the above definition of $\At(\sE^*)$ is completely equivalent to the one 
given in standard literature, especially \cite[Section 10.1]{HL} and \cite{BF}: therefore the Atiyah class is a homotopy invariant  and depends only on the isomorphism class  of the complex $\sE^*$ in the derived category of bounded complexes of locally free sheaves.

In particular, if $X$ is smooth projective and $\sE^* \to \sF$ is a finite locally free resolution of a coherent sheaf $\sF$, then the Atiyah class of $\sF$ is properly defined as $\At(\sF)=\At(\sE^*)$ and depends only on the class of $\sF$ in the bounded derived category of $X$.

\begin{lemma}\label{lem.dettagli} In the above setup,
	for every $a \in A^{*,*}_X(\HOM^*_{\Oh_X}(\sE^*,\sE^*))$ we have:
	\[ 
	[\delta + \debar,\nabla] (a)=[u,a],\qquad [\nabla (\debar), a] = \nabla(\debar a) + \debar \nabla (a)\,.\]
In particular, if $a$ is closed, then $[u,a]$ is exact. 
\end{lemma}

\begin{proof} 	We have  
	\begin{align*}
	&[u,a]= [[D^{1,0}, \delta + \debar], a] = [D^{1,0}, [\delta + \debar, a]] + [\delta + \debar, [D^{1,0}, a]]= [\delta + \debar, \nabla ] (a),
	\end{align*}
while for the last equality, by  \eqref{rem.debar}, 
	\[ [ \nabla(\debar ),a] = [[D^{1,0}, \debar], a] =[D^{1,0}, [\debar, a]] + [\debar, [D^{1,0}, a]] = [D - \debar, \debar a]+ \debar [D^{1,0}, a] = \nabla (\debar a) + \debar \nabla (a). \] 

\end{proof}

	The trace operator $\Tr \colon \HOM^*_{\Oh_X}(\sE^*, \sE^*) \to \Oh_X$ can be extended to 
	\[ \Tr \colon A^{*,*}_X (\HOM^*_{\Oh_X}(\sE^*, \sE^*)) \to A^{*,*}_X, \quad \Tr (\omega \cdot f) = \omega \Tr (f). \] 
	If we consider the DG-Lie algebra structure on $A^{*,*}_X (\HOM^*_{\Oh_X}(\sE^*, \sE^*))$ given above, with differential $[\delta, -] + \debar$, and endow $A^{*,*}_X$ with trivial bracket and differential $\debar$, then the trace operator is a morphism of DG-Lie algebras, so that
	\[ \Tr ([\delta, a] + \debar a) = \debar \Tr (a).\]
	
	\begin{lemma}\label{lem.tracciaconn} In the above setup, for every 
$h\in A^{*,*}_X(\HOM^*_{\Oh_X}(\sE^*,\sE^*))$ we have $\Tr ([D,h])=d \Tr(h)$.
	\end{lemma}
	
	\begin{proof} By linearity it is sufficient to consider the case 
		$h=\eta \cdot g$, with $ \eta \in \sA^{*,*}_X$ and $g \in \HOM^n_{\Oh_X}(\sE^*, \sE^*)$. 
It is clear that it is enough to consider $g$ of degree $0$, and by linearity we may assume  $g$ concentrated in one degree, i.e.,  $g = g_l \colon \sE^l \to \sE^l$. Let $e_1, \ldots, e_m$ be a local basis of holomorphic sections for $\sE^l$, and let
		\[ g (e_i) = \sum_j a_{ij} e_j, \quad D(e_i) = \sum_j \omega_{ij}  e_j, \quad \Tr (g)= (-1)^l \sum_i a_{ii}. \] 
		Then 
		\[ d \Tr (\eta \cdot g) = d (\eta \Tr (g)) =d\eta \Tr(g) + (-1)^{\overline{\eta}} \eta d\Tr(g), \]
		\begin{align*}
		& (\na  + \debar)( \eta \cdot g)(e_i) = D (\sum_j \eta a_{ij} e_j ) - (-1)^{\overline{\eta}} (\eta \cdot g) (\sum_j \omega_{ij}  e_j)\\
		&\quad=  \sum_j d(\eta) a_{ij} e_j   + \sum_j (-1)^{\overline{\eta}} \eta\wedge  d(a_{ij}) e_j  + \sum_{j}(-1)^{\overline{\eta}} \eta a_{ij} \wedge  D(e_j) - (-1)^{\overline{\eta}}\sum_{j} \eta \wedge \omega_{ij} g(e_j) \\
		&\quad  =\sum_k d(\eta) a_{ik} e_k   + \sum_k (-1)^{\overline{\eta}} \eta\wedge  d(a_{ik}) e_k   + \sum_{j,k}(-1)^{\overline{\eta}} \eta a_{ij} \wedge  \omega_{jk}e_k - (-1)^{\overline{\eta}}\sum_{j,k} \eta \wedge \omega_{ij} a_{jk}e_k\,.
		\end{align*}
		Therefore
		\begin{align*}
		\Tr ((\na  + \debar) (\eta \cdot g)) &= (-1)^l \sum_i  \left(d\eta a_{ii} + (-1)^{\overline{\eta}} \eta d(a_{ii}) + \sum_j \left( (-1)^{\overline{\eta}}\eta \wedge \omega_{ji}  a_{ij} - (-1)^{\overline{\eta}} \eta \wedge \omega_{ij} a_{ji}  \right) \right)\\
		& = (-1)^l \sum_i \left(d\eta a_{ii} + (-1)^{\overline{\eta}} \eta d(a_{ii})  \right) = d\eta \Tr(g)  + (-1)^{\overline{\eta}} \eta d \Tr(g). 
		\end{align*}
	\end{proof}

\begin{remark}\label{rem.nablade}
	It is useful to note that since the trace 
	is a morphism of differential graded Lie algebras,
	$ \Tr ([D, h]) = \Tr ((\na + \debar)(h))= d \Tr(h)$ of Lemma \ref{lem.tracciaconn} is equivalent to
	\[ \Tr (\nabla(h)) = \de \Tr (h). \] 
\end{remark}

\bigskip
\section{From connections and cyclic forms to $L_{\infty}$ morphisms}
\label{sec.Linfinito}

As in the previous section, let 
\[ \sE^*\colon\qquad 0\to \sE^p\xrightarrow{\,\delta\,}\sE^{p+1}\xrightarrow{\,\delta\,}\cdots\xrightarrow{\,\delta\,}\sE^{q}\to 0\]
be a fixed finite complex of locally free sheaves on a complex manifold $X$. 

\begin{definition} By a  cyclic (bilinear) form on the sheaf of DG-Lie algebras  
$\HOM^*_{\Oh_X}(\sE^*,\sE^*)$ we mean a 
graded symmetric $\Oh_X$-bilinear product of degree $0$
\[ \HOM^*_{\Oh_X}(\sE^*,\sE^*)\times \HOM^*_{\Oh_X}(\sE^*,\sE^*)\xrightarrow{\langle-,-\rangle}\Oh_X,\]
such that 
\[ \langle f,[g,h]\rangle=\langle [f,g],h\rangle\qquad \forall\; f,g,h\,.\]
\end{definition}

Equivalently, for every $f,g,h$ we have 
\[ \langle [f,g],h\rangle+(-1)^{\bar{f}\bar{g}}\langle g,[f,h]\rangle=0\]
i.e., $\langle -,-\rangle$ is invariant under the adjoint action. In particular
\begin{equation}\label{equ.deltaclosed} 
\langle [\delta,g],h\rangle+(-1)^{\bar{g}}\langle g,[\delta,h]\rangle=0\,.
\end{equation}
Notice that \eqref{equ.deltaclosed} is equivalent to the fact that the bilinear  form  
$\langle-,-\rangle$ is closed in the dual of $\HOM^*_{\Oh_X}(\sE^*,\sE^*)^{\odot 2}$.

Every cyclic form on $\HOM^*_{\Oh_X}(\sE^*,\sE^*)$  has a natural  extension
\[  A^{*,*}_X(\HOM^*_{\Oh_X}(\sE^*,\sE^*))^{\odot 2}\xrightarrow{\langle-,-\rangle}A^{*,*}_X,\qquad 
\langle\phi f,\psi g\rangle=(-1)^{\bar{f}\,\bar{\psi}}\phi\wedge \psi \langle f, g\rangle,\]
and it is immediate to check that, for $f,g \in A^{*,*}_X(\HOM^*_{\Oh_X}(\sE^*,\sE^*))$:  
\begin{equation} 
\debar \langle f,g \rangle =\langle \debar f, g \rangle + (-1)^{\overline{f} }\langle f, \debar g \rangle,
\end{equation}
and then $\langle-,-\rangle$ is $\debar+[\delta,-]$ closed.
Cyclic forms have received a lot of attention in several recent papers; for instance cyclic forms that are nondegenerate in cohomology play a central role in the  proof of the formality conjecture for polystable sheaves on projective surfaces with torsion canonical bundles, given in \cite{BMM2}. 

\begin{definition} We shall say that a connection $D$ of type $(1,0)$ on $\sE^*$ is compatible with the cyclic  form $ \langle -,- \rangle$ if 
\[ \langle [D,f],g\rangle+(-1)^{\bar{f}}\langle f,[D,g]\rangle=d\langle f,g\rangle\, ,\]
or equivalently if 
	\[ \langle\nabla (f),g\rangle+(-1)^{\bar{f}}\langle f,\nabla (g)\rangle=\de\langle f,g\rangle\, .\]
for every $f,g \in A^{*,*}_X(\HOM^*_{\Oh_X}(\sE^*,\sE^*))$.
\end{definition}

\begin{example}
According to  Lemma \ref{lem.tracciaconn} and Remark \ref{rem.nablade}, for every $a,b\in \C$ the form
\[\langle f,g\rangle=a\Tr(fg)+b\Tr(f)\Tr(g)\] 
is a cyclic form  of degree $0$ compatible with every connection of type (1,0).
\end{example}

We assume that the reader is familiar with the notion and basic properties of DG-Lie algebras and
$L_{\infty}$ morphisms between them, see e.g. \cite{BMM1,BMM2,fuka,getzler04,K,LMDT} and references therein. For the reader's convenience and to fix the sign convention, we only recall here the definition of an
$L_{\infty}$ morphism of DG-Lie algebras in the version that we use for 
explicit computations.

Let $V$ be a graded vector space over a field of characteristic 0. 
Given $v_1,\ldots,v_n$ 
homogeneous  vectors  of $V$ and a permutation $\sigma$ of $\{1,\ldots,n\}$, we denote by 
$\chi(\sigma;v_1,\ldots,v_n)=\pm 1$ the antisymmetric Koszul sign, defined by the relation 
\[ v_{\sigma(1)}\wedge\cdots\wedge v_{\sigma(n)}=\chi(\sigma;v_1,\ldots,v_n)\,  
 v_{1}\wedge\cdots\wedge v_{n}\,\]
in the $n$th exterior power $V^{\wedge n}$. We shall simply write $\chi(\sigma)$ instead of 
$\chi(\sigma;v_1,\ldots,v_n)$ when the vectors $v_1,\ldots,v_n$ are clear from the context.
For instance, if $\sigma$ is the transposition exchanging 1 and 2 we have 
$\chi(\sigma)=-(-1)^{\bar{v_1}\,\bar{v_2}}$. 
Notice that if every $v_i$ has odd degree, then 
$\chi(\sigma)=1$ for every $\sigma$.

Because of the universal property of wedge powers, we shall constantly interpret every linear map
$V^{\wedge p}\to W$ as a graded skew-symmetric $p$-linear map $V\times\cdots\times V\to W$.

\begin{definition}\label{def.morfismogeneral} 
Let $(V,\delta,[-,-])$ and  $(L,d,\{-,-\})$ be DG-Lie algebras over the same field. An 
$L_{\infty}$ morphism $g\colon V\rightsquigarrow L$ is a sequence of linear maps 
$g_n\colon V^{\wedge n}\to L$, $n\ge 1$, with $g_n$ of degree $1-n$ such that $g_1$ is a morphism of complexes, while for every 
$n\ge 2$ and every 
$v_1,\ldots,v_n\in V$ homogeneous we have
\[\begin{split}
\frac{1}{2}\sum_{p=1}^{n-1}&
\!\!\!\!\!\sum_{\quad\sigma\in S(p,n-p)}\!\!\!\!\!\!\chi(\sigma) (-1)^{(1-n+p)(|v_{\sigma(1)}|+\cdots+
|v_{\sigma(p)}|-p)}
\left\{g_p(v_{\sigma(1)},\ldots, v_{\sigma(p)}),\vphantom{\sum}
g_{n-p}(v_{\sigma(p+1)},\ldots,v_{\sigma(n)})\right\}\\
& +dg_n(v_1,\ldots,v_n)=
(-1)^{n-1}\sum_{\sigma\in S(1,n-1)}\chi(\sigma)g_{n}(\delta(v_{\sigma(1)}),v_{\sigma(2)},\ldots,v_{\sigma(n)})\\
&\qquad\qquad\qquad\qquad +(-1)^{n-2}\sum_{\sigma\in S(2,n-2)}\chi(\sigma)g_{n-1}([v_{\sigma(1)},v_{\sigma(2)}],v_{\sigma(3)},\ldots,v_{\sigma(n)}).\end{split}
\]
\end{definition}

Notice that the morphism of complexes $g_1$ factors to a morphism  $g_1\colon H^*(V)\to H^*(L)$, and the above condition for $n=2$, which is equivalent to   
\[ g_1([v_1,v_2])-
\{g_1(v_1),g_1(v_2)\}=dg_2(v_1,v_2)+g_2(\delta v_1,v_2)+(-1)^{\bar{v_1}}g_2(v_1,\delta v_2),\]
tells us that $g_1$ is a Lie morphism up to homotopy. In particular, the map $g_1\colon H^*(V)\to H^*(L)$
is a morphism of graded Lie algebras.

Conversely, given a morphism of graded Lie algebras $\tau \colon H^*(V)\to H^*(L)$ we shall say that an 
$L_{\infty}$ morphism $g\colon V\rightsquigarrow L$ lifts $\tau$ if  $g_1$ induces $\tau$ in cohomology.

In this paper we deal with $L_{\infty}$ morphisms where the target  $L$ is an abelian DG-Lie algebra: this means that $\{-,-\}=0$ and the above definition reduces to:

\begin{definition}\label{def.morfismo}
	Let $(V,\delta,[-,-])$ be a DG-Lie algebra and $(L,d)$ an abelian DG-Lie algebra. An $L_{\infty}$ morphism $g\colon V\rightsquigarrow L$ is a sequence of maps 
	$g_n\colon V^{\wedge n}\to L$, $n\ge 1$, with $g_n$ of degree $1-n$ such that the following conditions 
	$C_n$, $n=1,2,3,\ldots$, are satisfied:
	\begin{description}
	
	\item[$C_1$] $g_1\delta=d g_1$;\medskip
	
	\item[$C_n,\; n\ge 2$] for every 
	$v_1,\ldots,v_n\in V$ homogeneous we have
	\[ \begin{split}
	dg_n(v_1,\ldots,v_n)&=
	(-1)^{n-1}\sum_{\sigma\in S(1,n-1)}\chi(\sigma)g_{n}(\delta v_{\sigma(1)},v_{\sigma(2)},\ldots,v_{\sigma(n)})\\
	&\quad +(-1)^{n-2}\sum_{\sigma\in S(2,n-2)}\chi(\sigma)g_{n-1}([v_{\sigma(1)},v_{\sigma(2)}],v_{\sigma(3)},\ldots,v_{\sigma(n)}).\end{split}\]
	
	\end{description}
\end{definition}

Notice that if $g_n=0$ for every $n\ge N$ then $C_n$  is trivially satisfied for every 
$n>N$.
\medskip

We are now ready to prove the main result of this paper. 
In the following we consider 
 the shifted quotient $\dfrac{A^{*,*}_X}{A^{\ge 2,*}_X}[2]$ of the de Rham complex by the $2$nd subcomplex of the Hodge filtration as a DG-Lie algebra with  trivial bracket.

\begin{theorem}\label{thm.main1} Let $\sE^*$ be a finite complex of locally free sheaves on a complex manifold $X$ and let  $\langle-,-\rangle$ be a cyclic form of degree $0$ on $\HOM^*_{\Oh_X}(\sE^*,\sE^*)$ which is compatible with a connection $D$ of type $(1,0)$. Then 
	there is an $L_\infty$  morphism between DG-Lie algebras over the field $\C$
	\[ g\colon A^{0, *}_X(\HOM^*_{\Oh_X}(\sE^*,\sE^*))\rightsquigarrow \dfrac{A^{*,*}_X}{A^{\ge 2,*}_X}[2]\] 
	with components 
	\[\begin{split} g_1(f)&=\langle u,f\rangle=\langle f,u\rangle\in A^{1,*}_X[2],\\[4pt]
	g_2(f,g)&=\frac{1}{2}\left(\langle\nabla  (f),g\rangle-(-1)^{\overline{f}\overline{g}}
	\langle\nabla(g),f\rangle\right)\in A^{1,*}_X[2],\\[4pt]
	g_3(f,g,h)&=-\frac{1}{2}\langle f,[g,h]\rangle\in A^{0,*}_X[2],\end{split}\]
	and $g_n=0$ for every $n>3$.  As in the notation above, $u=\nabla (\debar + \delta)$ is the Atiyah cocycle of the connection $D$.
\end{theorem}

Notice  that the definition of $g$ only involves the DG-Lie structure of 
$A^{0, *}_X(\HOM^*_{\Oh_X}(\sE^*,\sE^*))$ and not the associative composition product.

\begin{proof}  Since the theorem gives  explicit formulas for the components $g_n$, the proof reduces to a straightforward computation. Since $g_n=0$ for every $n\ge 4$ 
	we need to check the conditions $C_n$ of Definition~\ref{def.morfismo} for $n=1,2,3,4$.
	For $C_1$ we have to prove  that
	\[  d g_1 (a) = g_1 ( [\delta, a] + \debar a ).\]
	This follows from the fact that  $[\delta, u] + \debar u =0 $ and that on the subcomplex $A^{1,*}_X \subseteq \dfrac{A^{*,*}_X}{A^{\ge 2,*}_X}$ we have $d= \debar$:
	\begin{align*}
	g_1([\delta, a] + \debar a) &=\langle u, [\delta, a] + \debar a \rangle =  -\langle [\delta, u],a \rangle +  \debar \langle u,a \rangle  -  \langle \debar u, a \rangle \\
	&  = -\langle [\delta, u] + \debar u, a \rangle + \debar \langle u,a \rangle = \debar \langle u,a \rangle = d g_1 (a).
	\end{align*}
	The condition $C_2$ is 
	\[ g_2 ([\delta, a_1] + \debar a_1, a_2) + (-1)^{\overline{a_1}} g_2 (a_1, [\delta, a_2] + \debar a_2) = g_1 ([a_1, a_2])-d g_2 (a_1, a_2).\]
	On the left hand side we have 
	\begin{align*}
	g_2 ([\delta, a_1] +& \debar a_1, a_2) + (-1)^{\overline{a_1}} g_2 (a_1, [\delta, a_2] + \debar a_2)  \\[3pt]
	& =\frac{1}{2} \Big( \langle \nabla ([\delta, a_1]), a_2 \rangle + \langle \nabla (\debar a_1), a_2 \rangle - (-1)^{\overline{a_1}\ \overline{a_2} + \overline{a_2}} \langle \nabla (a_2),  [\delta, a_1 ] \rangle  \\[3pt]
	&\quad - (-1)^{\overline{a_1}\ \overline{a_2} + \overline{a_2}} \langle \nabla (a_2),   \debar a_1 \rangle \Big) +
	\frac{1}{2} (-1)^{\overline{a_1}} \Big( \langle \nabla (a_1), [\delta, a_2]\rangle + \langle \nabla (a_1), \debar a_2\rangle \\[3pt]
	&\quad- (-1)^{\overline{a_1}\ \overline{a_2} + \overline{a_1}}  \langle \nabla ([\delta, a_2]), a_1 \rangle  - (-1)^{\overline{a_1}\ \overline{a_2} + \overline{a_1}}  \langle \nabla (\debar  a_2), a_1 \rangle \Big) \\[3pt]
	& =\frac{1}{2} \Big( \langle [\nabla (\delta), a_1]), a_2 \rangle - \langle [\delta, \nabla(a_1)]), a_2 \rangle + \langle \nabla (\debar a_1), a_2 \rangle - (-1)^{\overline{a_1}\ \overline{a_2} + \overline{a_2}} \langle \nabla (a_2),  [\delta, a_1 ] \rangle  \\[3pt]
	&\quad- (-1)^{\overline{a_1}\ \overline{a_2} + \overline{a_2}} \langle \nabla (a_2),   \debar a_1 \rangle  + (-1)^{\overline{a_1}} \langle \nabla (a_1), [\delta, a_2]\rangle + (-1)^{\overline{a_1}} \langle \nabla (a_1), \debar a_2\rangle  \\[3pt]
	&\quad- (-1)^{\overline{a_1}\ \overline{a_2}}  \langle [\nabla (\delta), a_2]), a_1 \rangle + (-1)^{\overline{a_1}\ \overline{a_2}}  \langle [\delta, \nabla(a_2)]), a_1 \rangle - (-1)^{\overline{a_1}\ \overline{a_2} }  \langle \nabla  (\debar  a_2), a_1 \rangle \Big)\\[3pt]
	& =\langle \nabla  (\delta), [a_1, a_2] \rangle + \frac{1}{2} \Big( \langle \nabla  (\debar a_1), a_2 \rangle - (-1)^{\overline{a_1}\ \overline{a_2} + \overline{a_2}} \langle \nabla (a_2),   \debar a_1 \rangle   +  (-1)^{\overline{a_1}} \langle \nabla (a_1), \debar a_2\rangle \\[3pt]
	&\quad - (-1)^{\overline{a_1}\ \overline{a_2} }  \langle \nabla (\debar  a_2), a_1 \rangle \Big). 
	\end{align*}
	
Using Lemma \ref{lem.dettagli},  the right hand side is:
	
	\begin{align*}
	g_1 ([a_1, a_2])& - d g_2 (a_1, a_2) = \langle \nabla (\delta) +\nabla (\debar), [a_1, a_2]  \rangle - \frac{1}{2} \Big( \langle \debar \nabla (a_1), a_2\rangle \\[3pt]
	&\quad - (-1)^{\overline{a_1}} \langle \nabla (a_1), \debar a_2 \rangle - (-1)^{\overline{a_1}\ \overline{a_2}} \langle \debar \nabla (a_2), a_1 \rangle + (-1)^{\overline{a_1}\ \overline{a_2} + \overline{a_2}} \langle \nabla  (a_2), \debar a_1\rangle \Big)  \\[3pt]
	& =\langle \nabla (\delta), [a_1, a_2]  \rangle - \frac{1}{2} \Big( \langle \debar \nabla (a_1), a_2\rangle - (-1)^{\overline{a_1}} \langle \nabla(a_1), \debar a_2 \rangle - (-1)^{\overline{a_1}\ \overline{a_2}} \langle \debar \nabla (a_2), a_1 \rangle \\[3pt]
	&\quad + (-1)^{\overline{a_1}\ \overline{a_2} + \overline{a_2}} \langle \na (a_2), \debar a_1\rangle \Big) + \frac{1}{2} \Big(\langle [\na(\debar),a_1], a_2 \rangle - (-1)^{\overline{a_1}\ \overline{a_2}}\langle [\na(\debar),a_2], a_1 \rangle\Big) \\[3pt]
	&= \langle \nabla (\delta), [a_1, a_2]  \rangle + \frac{1}{2} \Big( - \langle \debar \nabla (a_1), a_2\rangle + (-1)^{\overline{a_1}} \langle \nabla (a_1), \debar a_2 \rangle + (-1)^{\overline{a_1}\ \overline{a_2}} \langle \debar \nabla (a_2), a_1 \rangle \\[3pt]
	&\quad - (-1)^{\overline{a_1}\ \overline{a_2} + \overline{a_2}} \langle \na (a_2), \debar a_1\rangle + \langle \na(\debar a_1), a_2 \rangle - (-1)^{\overline{a_1}\ \overline{a_2}}\langle \na(\debar a_2), a_1 \rangle + \langle \debar \na (a_1), a_2\rangle \\[3pt]
	&\quad - (-1)^{\overline{a_1}\ \overline{a_2} }\langle \debar \na (a_2), a_1\rangle \Big)\\[3pt]
	& =\langle \nabla (\delta), [a_1, a_2]  \rangle + \frac{1}{2} \Big(  (-1)^{\overline{a_1}} \langle \nabla (a_1), \debar a_2 \rangle  - (-1)^{\overline{a_1}\ \overline{a_2} + \overline{a_2}} \langle \na (a_2), \debar a_1\rangle \\[3pt]
	&\quad  + \langle \na(\debar a_1), a_2 \rangle - (-1)^{\overline{a_1}\ \overline{a_2}}\langle \na(\debar a_2), a_1 \rangle \Big),
	\end{align*}
and this proves $C_2$. For $C_3$ we need to check that
	\begin{align*}
	dg_3 (a_1, a_2, a_3) &= g_3([\delta, a_1] + \debar a_1, a_2, a_3) - (-1)^{\overline{a_1}\ \overline{a_2}} g_3([\delta, a_2] + \debar a_2, a_1, a_3) \\
	& \quad + (-1)^{\overline{a_3}(\overline{a_1}+ \overline{a_2})} g_3 ([\delta, a_3] + \debar a_3, a_1, a_2) - g_2([a_1, a_2], a_3) + (-1)^{\overline{a_2}\ \overline{a_3}} g_2 ([a_1, a_3], a_2) \\
	&\quad - (-1)^{\overline{a_1} ( \overline{a_2}+ \overline{a_3})} g_2 ([a_2, a_3], a_1).
	\end{align*}

Using the compatibility of the connection and the cyclic form, the terms involving $g_2$ can be expanded as:
	\begin{align*}
	&- g_2([a_1, a_2], a_3) + (-1)^{\overline{a_2}\ \overline{a_3}} g_2 ([a_1, a_3], a_2)- (-1)^{\overline{a_1} ( \overline{a_2}+ \overline{a_3})} g_2 ([a_2, a_3], a_1) \\[3pt]
	&  =-\frac{1}{2} \big( \langle \nabla ([a_1, a_2]), a_3 \rangle - (-1)^{\overline{a_3}(\overline{a_1}+ \overline{a_2})} \langle \nabla (a_3), [a_1, a_2] \rangle \big) \\[3pt]
	&\quad + \frac{1}{2}(-1)^{\overline{a_2}\ \overline{a_3}} \big(\langle \nabla ([a_1, a_3]), a_2 \rangle  - (-1)^{\overline{a_2}(\overline{a_1}+ \overline{a_3})} \langle \nabla(a_2), [a_1, a_3] \rangle \big) \\[3pt]
	&\quad- \frac{1}{2} (-1)^{\overline{a_1}(\overline{a_2}+ \overline{a_3})} \big(\langle \nabla ([a_2, a_3]),a_1 \rangle  - (-1)^{\overline{a_1}(\overline{a_2}+ \overline{a_3})} \langle \nabla(a_1), [a_2, a_3] \rangle \big) \\[3pt]
	&= -\frac{1}{2} \Big( \langle [\nabla(a_1),a_2], a_3 \rangle + (-1)^{\overline{a_1}} \langle [a_1,\nabla(a_2)], a_3 \rangle -  (-1)^{\overline{a_3}(\overline{a_1}+ \overline{a_2})} \langle \nabla (a_3), [a_1, a_2] \rangle \\[3pt]
	&\quad - (-1)^{\overline{a_2}\ \overline{a_3}} \langle [\nabla(a_1), a_3], a_2 \rangle  - (-1)^{\overline{a_2}\ \overline{a_3} + \overline{a_1}} \langle [a_1, \nabla(a_3) ], a_2 \rangle \\[3pt]
	&\quad + (-1)^{\overline{a_1}\ \overline{a_2}} \langle \nabla(a_2), [a_1, a_3] \rangle + (-1)^{\overline{a_1}(\overline{a_2}+ \overline{a_3})} \langle [\nabla(a_2), a_3], a_1 \rangle  \\[3pt]
	&\quad + (-1)^{\overline{a_1}(\overline{a_2}+ \overline{a_3})+ \overline{a_2}} \langle [a_2, \nabla(a_3)], a_1 \rangle - \langle  \nabla (a_1) , [a_2, a_3] \rangle\Big) = -\frac{1}{2} \de \langle a_1, [a_2, a_3] \rangle.
	\end{align*}
	On the other hand,
	\begin{align*}
	& g_3([\delta, a_1] + \debar a_1, a_2, a_3) - (-1)^{\overline{a_1}\ \overline{a_2}} g_3([\delta, a_2] + \debar a_2, a_1, a_3)  + (-1)^{\overline{a_3}(\overline{a_1}+ \overline{a_2})} g_3 ([\delta, a_3] + \debar a_3, a_1, a_2) \\
	&\qquad= - \frac{1}{2} \big( \langle [\delta, a_1],[a_2, a_3] \rangle + (-1)^{\overline{a_1}} \langle a_1, [[\delta, a_2], a_3 ] \rangle + (-1)^{\overline{a_1} + \overline{a_2}} \langle a_1, [a_2, [\delta, a_3]] \rangle + \langle \debar a_1,[a_2, a_3] \rangle\\
	&\qquad\quad+ (-1)^{\overline{a_1}} \langle a_1, [\debar a_2, a_3 ] \rangle + (-1)^{\overline{a_1} + \overline{a_2}} \langle a_1, [a_2, \debar a_3] \rangle\big) =  - \frac{1}{2} \debar \langle a_1, [a_2, a_3] \rangle
	\end{align*}
	so that we obtain
	\[   dg_3(a_1,[a_2,a_3])=-\frac{1}{2} d \langle a_1, [a_2, a_3] \rangle =  -\frac{1}{2} \debar \langle a_1, [a_2, a_3] \rangle -\frac{1}{2} \de \langle a_1, [a_2, a_3] \rangle.\] 
	
	Lastly, the condition  $C_4$ is
	\begin{align*}
	&g_3 ([a_1, a_2], a_3, a_4) - (-1)^{\overline{a_2}\ \overline{a_3}} g_3 ([a_1, a_3], a_2, a_4 ) + (-1)^{\overline{a_4}(\overline{a_2} + \overline{a_3})} g_3 ([a_1, a_4], a_2, a_3)\\[3pt]
	&\qquad + (-1)^{\overline{a_1}(\overline{a_2} + \overline{a_3})} g_3 ([a_2, a_3], a_1, a_4) - (-1)^{\overline{a_3}\ \overline{a_4} + \overline{a_1}\ \overline{a_2} + \overline{a_1}\ \overline{a_4}} g_3 ([a_2, a_4], a_1, a_3)  \\[3pt]
	&\qquad +(-1)^{(\overline{a_1} + \overline{a_2})(\overline{a_3}+ \overline{a_4})}g_3 ([a_3, a_4], a_1, a_2)=0.
	\end{align*}
	We have that 
	\begin{align*}
	\frac{1}{2}&\langle [a_1, a_2], [a_3, a_4] \rangle - (-1)^{\overline{a_2}\ \overline{a_3}} \frac{1}{2}\langle [a_1, a_3], [a_2, a_4 ] \rangle  + (-1)^{\overline{a_4}(\overline{a_2} + \overline{a_3})} \frac{1}{2}\langle [a_1, a_4], [a_2, a_3] \rangle \\[3pt]
	&\quad + (-1)^{\overline{a_1}(\overline{a_2} + \overline{a_3})} \frac{1}{2}\langle [a_2, a_3], [a_1, a_4]\rangle  - (-1)^{\overline{a_3}\ \overline{a_4} + \overline{a_1}\ \overline{a_2} + \overline{a_1}\ \overline{a_4}} \frac{1}{2}\langle  [a_2, a_4], [a_1, a_3] \rangle   \\[3pt]
	&\quad +(-1)^{(\overline{a_1} + \overline{a_2})(\overline{a_3}+ \overline{a_4})}\frac{1}{2}\langle [a_3, a_4],[a_1, a_2]\rangle \\[3pt]
	&= \langle [a_1, a_2], [a_3, a_4] \rangle - (-1)^{\overline{a_2}\ \overline{a_3}} \langle [a_1, a_3], [a_2, a_4 ] \rangle + (-1)^{\overline{a_4}(\overline{a_2} + \overline{a_3})} \langle [a_1, a_4], [a_2, a_3] \rangle\\[3pt]
	&= \langle a_1, [a_2, [a_3, a_4]] \rangle - (-1)^{\overline{a_2}\ \overline{a_3}} \langle a_1, [a_3, [a_2, a_4 ]] \rangle + (-1)^{\overline{a_4}(\overline{a_2} + \overline{a_3})} \langle a_1, [a_4, [a_2, a_3]] \rangle \\[3pt]
	& = \langle a_1, [a_2, [a_3, a_4]] -(-1)^{\overline{a_2}\ \overline{a_3}} [a_3, [a_2, a_4 ]] -  [[a_2, a_3], a_4]\rangle =0\,.
	\end{align*}
\end{proof}

\begin{corollary}\label{cor.main1} 
Let $\sE^*$ be a finite complex of locally free sheaves on a complex manifold $X$. Then 
every connection $D$ of type $(1,0)$ on $\sE^*$ gives an 
$L_\infty$  morphism between DG-Lie algebras on the field $\C$
\[ g\colon A^{0, *}_X(\HOM^*_{\Oh_X}(\sE^*,\sE^*))\rightsquigarrow \dfrac{A^{*,*}_X}{A^{\ge 2,*}_X}[2]\] 
with components 
\[\begin{split} g_1(f)&=-\Tr(uf)\in A^{1,*}_X[2]\\[4pt]
	g_2(f,g)&=-\frac{1}{2}\Tr \left(\nabla(f)g-(-1)^{\overline{f}\overline{g}}
	\nabla(g)f\right)\in A^{1,*}_X[2]\\[4pt]
	g_3(f,g,h)&=\frac{1}{2}\Tr(f[g,h])\in A^{0,*}_X[2],\end{split}\]
	and $g_n=0$ for every $n>3$.
\end{corollary}

\begin{proof} Use the cyclic form  $\langle f,g\rangle=-\Tr(fg)$ in Theorem~\ref{thm.main1}.\end{proof}

\begin{remark}
The homotopy class of the $L_{\infty}$ morphism $g$ of  Corollary~\ref{cor.main1} depends on the choice of the connection. This also holds for complex analytic  connections, that is for connections where the Atiyah cocycle vanishes $u=0$ and therefore $g_1=0$. This implies that 
$g_2$ factors to a bilinear graded skewsymmetric map  in cohomology 
\[ g_2\colon \Ext^i_X(\sE^*,\sE^*)\times \Ext^j_X(\sE^*,\sE^*)\to \mathbb{H}^{i+j+1}
\left(\dfrac{A^{*,*}_X}{A^{\ge 2,*}_X}\right)\]
that depends only on the homotopy class of $g$.

In order to see that the above maps depend on the connection it is sufficient to consider the example of a trivial bundle of rank $2$ over an elliptic curve $X$.
In this case, since $\Omega^1_X$ is trivial, every  complex analytic connection is of type $D=d+\theta$, where 
$\theta$ is a $2\times 2$ matrix with values in $H^0(X,\Omega_X^1)$ and then 
$\nabla=\debar+[\theta,-]$. Similarly $\Ext^0_X(\sE^*,\sE^*)$ is identified with the Lie algebra $M_{2,2}(\C)$ of 
$2\times 2$ matrices with constant  coefficients and therefore 
\[ g_2(a,b)=-\frac{1}{2}([\theta,a]b-[\theta,b]a)\in H^0(\Omega^1_X)=\mathbb{H}^{1}
\left(\dfrac{A^{*,*}_X}{A^{\ge 2,*}_X}\right)\,.\]
If $dz$ is a generator of $H^0(X,\Omega^1_X)$ and $\theta=Cdz$, with $C\in M_{2,2}(\C)$, 
the conclusion follows by observing that the rank of the bilinear map 
\[ M_{2,2}(\C)\times M_{2,2}(\C)\to \C,\quad (A,B)\mapsto \frac{1}{2}\Tr([C,A]B-[C,B]A)=\Tr(C[A,B]),\]
is equal to $0$  when $C$ is a multiple of the identity and  is  2 otherwise.
\end{remark}

\bigskip
\section{Semiregularity and deformations of coherent sheaves}

Let $\sF$ be a coherent sheaf on a complex manifold $X$ equipped with a finite  locally free resolution 
\begin{equation} \label{equ.resolution}
0\to \sE^{-n}\xrightarrow{\,\delta\,} \cdots\xrightarrow{\,\delta\,} \sE^0\to \sF\to 0\,.
\end{equation}

Then the deformation theory of $\sF$ is controlled by the DG-Lie algebra 
$A_X^{0,*}(\HOM^*_{\Oh_X}(\sE^*,\sE^*))$ defined in the previous sections for arbitrary finite complexes of locally free sheaves. This means that, over a local Artin $\C$-algebra $A$, the deformations of $\sF$ over $A$ are determined by solutions of the Maurer-Cartan equations 
\[ \debar x+[\delta,x]+\frac{1}{2}[x,x]=0,\qquad x\in 
\bigoplus_{i\ge 0} A_X^{0,i}(\HOM^{1-i}_{\Oh_X}(\sE^*,\sE^*)\otimes \mathfrak{m}_A,\]
with $\mathfrak{m}_A$ the maximal ideal of $A$. Moreover two solutions of the Maurer-Cartan equation give isomorphic deformations if and only if they are gauge equivalent: when $\sF$ is locally free this fact is nowadays standard \cite{fuka,LMDT} and extends quite easily to the general case, see e.g. \cite{BMM2,FIM,Meazz}.

The resolution \eqref{equ.resolution} can also be used for an explicit description of the (modified)  semiregularity map 
\[  \tau=\sum_{p\ge 0}\tau_p\colon 
\Ext_X^*(\sF,\sF)\to \bigoplus_{p\ge 0}\mathbb{H}^{*}(X,\Omega_X^{\le p}[2p]),\qquad \tau(a)=\Tr(\exp(-\At(\sF))\circ a)\,.\]

Since the quasi-isomorphic complexes $\sE^*$ and $\sF$ have the same Atiyah class we have 
$\At(\sF)=[u]$, where $u\in A_X^{1,*}(\HOM^*_{\Oh_X}(\sE^*,\sE^*))$ is the Atiyah cocycle of a connection of type $(1,0)$ on $\sE^*$. Moreover
the hypercohomology of $\Omega_X^{\le p}[2p]$ is computed by the truncated de Rham complex $A_X^{*,*}/A_X^{>p,*}[2p]$, and then every component $\tau_p$ is induced in cohomology by the morphism of complexes
\[ A_X^{0,*}(\HOM^*_{\Oh_X}(\sE^*,\sE^*))\to 
\frac{A_X^{*,*}}{A_X^{>p,*}}[2p],\qquad f\mapsto \frac{(-1)^p}{p!}\Tr(u^pf)\,.\]

Then, Corollary~\ref{cor.main1} immediately gives the following theorem.

\begin{corollary}\label{cor.thm.main2} Let $\sF$ be a coherent sheaf on a complex manifold $X$ equipped with a finite locally free resolution $\sE^*$. Then every connection of type $(1,0)$ on the resolution $\sE^*$ gives a lifting of 
\[  \tau_1\colon \Ext_X^*(\sF,\sF)\to \mathbb{H}^{*}(X,\Omega_X^{\le 1}[2])\]
to an $L_{\infty}$ morphism 
\[ g\colon A_X^{0,*}(\HOM^*_{\Oh_X}(\sE^*,\sE^*))\rightsquigarrow \frac{A_X^{*,*}}{A_X^{>1,*}}[2]\,.\]
\end{corollary}

Now the application to deformation theory of Corollary~\ref{cor.thm.main2} follows by a completely standard argument, the same used and explained, for instance, in \cite{algebraicBTT,K,EDF,ManRendiconti,LMDT}: every $L_{\infty}$ morphism  
$g\colon V\rightsquigarrow L$  of DG-Lie algebras induces a morphism of deformation functors $\Def_V\to \Def_L$ (here $\Def$ means the functor of Maurer-Cartan solution modulus gauge action) such that 
the induced map in cohomology  commutes with obstruction maps. If $L$ is abelian, then every obstruction in $\Def_L$ is trivial, hence every obstruction in $\Def_V$ belongs to the kernel of $g_1\colon H^2(V)\to H^2(L)$.

\begin{corollary}\label{cor.main2} Let $\sF$ be a coherent sheaf on a complex manifold $X$ admitting a locally free resolution. Then every obstruction to the deformations of $\sF$ belongs to the kernel of the map
\[  \tau_1\colon \Ext_X^2(\sF,\sF)\to \mathbb{H}^{2}(X,\Omega_X^{\le 1}[2]).\]
If the Hodge to de Rham spectral sequence of $X$ degenerates at $E_1$, then  
every obstruction to the deformations of $\sF$ belongs to the kernel of the map
\[  \sigma_1\colon \Ext_X^2(\sF,\sF)\to {H}^{3}(X,\Omega_X^{1}),\qquad \sigma_1(a)=-\Tr(\At(\sF)\circ a).\]
\end{corollary}

\begin{proof} By the syzygy theorem it is not restrictive to assume that $\sF$ admits a finite locally free resolution
$\sE^*$.
According to  Corollary~\ref{cor.thm.main2} the map $\tau_1$ lifts to 
an $L_{\infty}$ morphism 
\[ g\colon A_X^{0,*}(\HOM^*_{\Oh_X}(\sE^*,\sE^*))\rightsquigarrow \frac{A_X^{*,*}}{A_X^{\ge 2,*}}[2]\,\]
and we have that the linear component $g_1$ commutes with obstruction maps of the associated deformation functors. 
By construction the DG-Lie algebra 
$\dfrac{A_X^{*,*}}{A_X^{\ge 2,*}}[2]$ has trivial bracket and hence every obstruction of the associated deformation functor is trivial.

If the Hodge to de Rham spectral sequence of $X$ degenerates at $E_1$ then the inclusion 
of complexes $A^{1,*}_X[2]\subset \dfrac{A_X^{*,*}}{A_X^{\ge 2,*}}[2]$ is injective in cohomology
\[ H^3(X,\Omega^1_X)\hookrightarrow \mathbb{H}^{2}(X,\Omega_X^{\le 1}[2])\]
and the maps $\sigma,\tau$ have the same kernel.
\end{proof}

\bigskip
\section{Outline of the analogous algebraic construction}
\label{sec.outline}

Since the proof of the main theorem is mostly algebraic, it is not surprising the same can be slightly modified in order to have an algebraic analogue, valid on a smooth separated scheme $X$ of finite type over a field $\K$ of characteristic 0,  provided that an algebraic representative for $R\Hom(\sF,\sF)$ is given. Here we give only a sketch of a possible algebraic proof,  a more detailed description is given in the paper by the first author \cite{lepri}.

Let $\sE^*$ be a finite complex of locally free sheaves on $X$ and let $\sU=\{U_i\}$ be an open 
affine cover of $X$. For simplicity of exposition we assume that $\sE^*$ is a complex of free sheaves on every $U_i$, although this additional assumption is unnecessary and it can be easily removed. 

Thus, according to \cite{FIM}, a possible DG-Lie algebra representing  $R\Hom(\sF,\sF)$ is 
the totalization $\Tot(\sU,\HOM^*_{\Oh_X}(\sE^*,\sE^*))$
of the cosimplicial DG-Lie algebra of \v{C}ech cochains in $\HOM^*_{\Oh_X}(\sE^*,\sE^*)$ with respect to the open cover $\sU$, see also \cite{Meazz}: here we follow the notation of \cite{DMcoppie,LMDT}, where the reader can also find a complete and explicit definition of the totalization functor 
\[ \Tot\colon\{\text{cosimplicial DG-vector spaces}\}\to 
\{\text{DG-vector spaces}\}\,\]
together with its main properties. Here we only recall that $\Tot$ preserves possible multiplicative structures, hence transforms cosimplicial DG-Lie algebras (resp.: cosimplicial abelian DG-Lie algebras) into 
DG-Lie algebras (resp.: abelian DG-Lie algebras).

Denote by $\Omega=\Omega_{X/\K}$ the sheaf of K\"{a}hler differentials.
For every coherent sheaf $\sM$ we denote by $\DER_{\K}(\Oh_X,\sM)\simeq\HOM_{\Oh_X}(\Omega,\sM)$ the sheaf of  $\K$-linear 
derivations $\Oh_X\to \sM$, considered as a complex concentrated in degree $0$. 
Following \cite{DMcoppie} we define the graded sheaf 
\[ \sJ^*_{\sM}=\{(f,\alpha)\in \HOM^*_{\K}(\sE^*,\sM\otimes \sE^*)\times \DER_{\K}(\Oh_X,\sM)\mid 
f(ax)=af(x)+\alpha(a)\otimes x,\; \forall x\in \sE^*,\, a\in \Oh_X\}.\]
The same argument of \cite{DMcoppie} shows that $\sJ^*_{\sM}$ is a finite complex of coherent sheaves and there exists a short exact sequence of  complexes 
\[ 0\to  \HOM^*_{\Oh_X}(\sE^*,\sM\otimes_{\Oh_X} \sE^*)\xrightarrow{f\mapsto (f,0)} \sJ^*_{\sM}\xrightarrow{(f,\alpha)\mapsto\alpha} 
\DER_{\K}(\Oh_X,\sM)\to 0\,.\]

Then,  (a germ of) an algebraic connection on $\sE^*$ may be conveniently defined as an element of $\sJ^0_{\Omega}$  mapped onto the  universal derivation $d\colon \Oh_X\to \Omega$. 
Clearly, since the map  $\sJ^0_{\Omega}\to \DER_{\K}(\Oh_X,\Omega)$ is generally not surjective on global sections, a global algebraic connection on $\sE^*$ does not necessarily exist. 

However, a global algebraic connection always exists in the totalization of $\sJ^*_{\Omega}$ with respect to the affine open cover $\sU$. 
In fact, the exact sequence of coherent sheaves 
\[ 0\to  \HOM^*_{\Oh_X}(\sE^*,\Omega\otimes_{\Oh_X} \sE^*)\xrightarrow{f\mapsto (f,0)} 
\sJ^*_{\Omega}\xrightarrow{(f,\alpha)\mapsto\alpha} 
\DER_{\K}(\Oh_X,\Omega)\to 0\,\]
gives a short exact sequence of the corresponding cosimplicial complexes of \v{C}ech cochains in the open affine cover $\sU$; since $\Tot$ is an exact functor (see e.g. \cite{FHT,LMDT}) we get  an exact sequence 
\[ 0\to  \Tot(\sU,\HOM^*_{\Oh_X}(\sE^*,\Omega\otimes_{\Oh_X} \sE^*))\xrightarrow{\quad} 
\Tot(\sU,\mathcal{J}_{\Omega}^*)\xrightarrow{\quad } 
\Tot(\sU,\DER_{\K}(\Oh_X,\Omega))\to 0\,.\]
In view of the natural inclusion of  global sections into the totalization, the universal derivation $d\colon \Oh_X\to \Omega$ belong to  $\Tot(\sU,\DER_{\K}(\Oh_X,\Omega))$ and we may define a \emph{connection of type $(1,0)$ on $\sE^*$} as   an element 
$D\in \Tot(\sU,\mathcal{J}^*_{\Omega})$ mapped onto $d$.

Now everything works, mutatis mutandis, as in the previous sections: consider the complex 
$\Oh_X[2]\xrightarrow{d}\Omega[1]$ as a sheaf of abelian DG-Lie algebras and
define the Atiyah cocycle 
\[u\in \Tot(\sU,\HOM^*_{\Oh_X}(\sE^*,\Omega\otimes_{\Oh_X} \sE^*))\] 
 as the differential of $D$. Denote 
by \[\nabla=[D,-]\colon \Tot(\sU,\HOM^*_{\Oh_X}(\sE^*,\sE^*))\to \Tot(\sU,\HOM^*_{\Oh_X}(\sE^*,\Omega\otimes_{\Oh_X} \sE^*))\] 
the adjoint of $D$ and use the same formulas of Theorem~\ref{thm.main1} in order to define 
an $L_{\infty}$ morphism 
\[ g\colon \Tot(\sU,\HOM^*_{\Oh_X}(\sE^*,\sE^*))\rightsquigarrow 
\Tot(\sU,\Oh_X[2]\xrightarrow{d}\Omega[1])\,. \]
Finally, by Whitney's integration theorem \cite{getzler04,LMDT} the cohomology 
of  $\Tot(\sU,\Oh_X[2]\xrightarrow{d}\Omega[1])$ is the same as the hypercohomology of 
$\Oh_X[2]\xrightarrow{d}\Omega[1]$ and 
the linear component $g_1$ induces in cohomology the first component $\tau_1$ of the (modified) semiregularity map.

\begin{ackno} We thank Francesco Meazzini and Ruggero Bandiera for useful discussions on the subject of this paper. Our thanks also to the anonymous referee for several useful remarks and comments. 
\end{ackno}

\end{document}